\documentclass[12pt]{amsart}
\usepackage{amsfonts, amssymb, latexsym, graphicx, hyperref, amsmath,enumitem,amsthm}
\usepackage{color}
\usepackage{tikz}
\usepackage{mathtools}
\usepackage[vcentermath]{youngtab}
\usepackage{stackengine}
\usepackage{ytableau}
\ytableausetup{centertableaux}
\usepackage{algorithm,algorithmic}
\usetikzlibrary{calc, shapes, backgrounds,arrows,positioning,decorations.pathmorphing}

\newtheorem{thm}{Theorem}[section]

\newtheorem{lem}{Lemma}[section]
\newtheorem{prop}{Proposition}[section]
\newtheorem{cor}{Corollary}[section]

\theoremstyle{remark}
\newtheorem{ex}[thm]{Example}
\newcommand{\cost}{{\sf cost}}
\newcommand{\width}{{\sf width}}

\title{Kostka semigroups and generalized Dyck paths}
\author{Jaehyung Kim }
\address{Dept.~of Mathematics, U. Illinois at Urbana-Champaign, Urbana, IL 61801, USA}
\email{jk65@illinois.edu}
\date{May 14, 2021}

\begin{document}
\pagestyle{plain}

\maketitle
\begin{abstract}
We prove a conjecture of S.~Gao-J.~Kiers-G.~Orelowitz-A.~Yong which asserts the reducibility
of certain generalized Dyck paths. This gives a strengthening, and new proof, for their 
Width Bound Theorem
on the Hilbert basis of the Kostka semigroup.
\end{abstract}

\section{Introduction}
Let $\vec{x} = (x_1, \ldots , x_t)$ be a list of nonzero integers. S.~Gao, J.~Kiers, G.~Orelowitz, and A.~Yong  \cite{hilbasis} define $\vec{x}$ to be \emph{generalized Catalan} if
\begin{equation} \label{eq:1}
\sum_{i=1}^t x_i = 0    \text{ \ and \ }
    \sum_{i=1}^q x_i \geq 0 \text{\ \  for  all \ $1 \leq q \leq t$.}
\end{equation}
Furthermore,  $\vec{x}$ is \emph{reducible} if there is a (generalized) Catalan sublist $\vec{x}^{\circ} = (x_{i_1}, x_{i_2}, \ldots , x_{i_a})$ such that the complementary sublist $\vec{x}^{\bullet}$ is also Catalan. 

A maximal consecutive sublist of $\vec{x}$ consisting of integers of the same sign is a \emph{run}. By (\ref{eq:1}) there are even number $2y$ of runs. Let $a_k > 0$ be the maximum (in absolute value) of any $x_i$ in run $k$. Define
\begin{equation}
    \cost(\vec{x}) = \sum_{k=1}^{2y} a_k \text{ and $\width(\vec{x}) = t$.}
\end{equation}

Our main result is a proof of \cite[Conjecture~5.3]{hilbasis}:
\begin{thm} \label{conj1}
  If $\cost(\vec{x}) < \width(\vec{x})$ then $\vec{x}$ is reducible.
\end{thm}

\begin{ex} \label{ex1}
The following list is generalized Catalan:
\[\vec{x} = (\underline{5},5,4,4, \underline{-3},-3,-3,-3,-3, \underline{-1}, 5,5,5,\underline{3}, \underline{-4},-4,-4,-4,-4).\] 
One has
\[\cost(\vec{x}) = 5+3+5+4 = 17, \text{ \ and \ $\width(\vec{x}) = 19$.}\]
Furthermore, let $\vec{x}^{\circ} = (5,-3,-1,3,-4)$ be a sublilst consisting of underlined elements, then $\vec{x}^{\circ}$ and the complementary sublist $\vec{x}^{\bullet} = (5,4,4,-3,-3,-3,-3,5,5,5,-4,-4,-4,-4)$ witness that $\vec{x}$ is reducible.\qed

It is sometimes useful to view $\vec{x}$ as a (generalized) \emph{Dyck path}, where $q$th edge is of length $|x_q|\sqrt{2}$. See Figure \ref{fig1} for such a visualization for $\vec{x}$
from Example~\ref{ex1}; blue edges represent the sublist $\vec{x}^{\circ}$ and gray edges represent the sublist $\vec{x}^{\bullet}$.

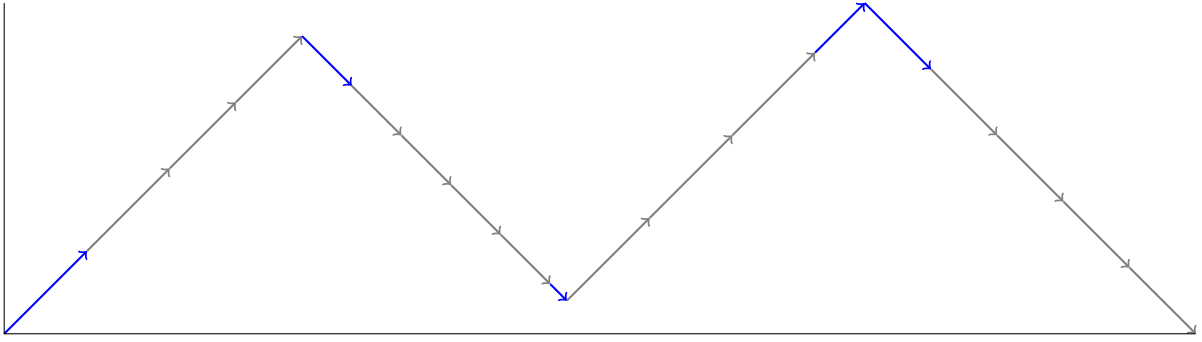
\begin{figure} 
\begin{tikzpicture}[scale = 0.22]
\draw [->, blue, thick] (0,0) -- (5,5);
\draw [->, gray, thick] (5,5) -- (10,10);
\draw [->, gray, thick] (10,10) -- (14,14);
\draw [->, gray, thick] (14,14) -- (18,18);
\draw [->, blue, thick] (18,18) -- (21,15);
\draw [->, gray, thick] (21,15) -- (24,12);
\draw [->, gray, thick] (24,12) -- (27,9);
\draw [->, gray, thick] (27,9) -- (30,6);
\draw [->, gray, thick] (30,6) -- (33,3);
\draw [->, blue, thick] (33,3) -- (34,2);
\draw [->, gray, thick] (34,2) -- (39,7);
\draw [->, gray, thick] (39,7) -- (44,12);
\draw [->, gray, thick] (44,12) -- (49,17);
\draw [->, blue, thick] (49,17) -- (52,20);
\draw [->, blue, thick] (52,20) -- (56,16);
\draw [->, gray, thick] (56,16) -- (60,12);
\draw [->, gray, thick] (60,12) -- (64,8);
\draw [->, gray, thick] (64,8) -- (68,4);
\draw [->, gray, thick] (68,4) -- (72,0);
\draw (0,20) -- (0,0) -- (72,0);

\end{tikzpicture}
\caption{Dyck path for Example \ref{ex1}.}
\label{fig1}
\end{figure}

\end{ex}

The motivation for Theorem~\ref{conj1} comes from the study of the Hilbert basis of the \emph{Kostka 
semigroup} presented in \cite{hilbasis}. Let ${\sf Par}_r(n)$ be the set of integer partitions $\lambda=(\lambda_1\geq\lambda_2\geq\ldots
\geq\lambda_r\geq 0)$ of $n$ with at most $r$ nonzero parts. (We identify integer partitions with their Young diagrams). Recall that if $\lambda,\mu\in {\sf Par}_r(n)$ then $\lambda\geq \mu$ in \emph{dominance order} if $\sum_{i=1}^t \lambda_i\geq \sum_{i=1}^t \mu_i$ for $1\leq t\leq r$.

The \emph{Kostka cone} is defined as :
\[ \sf{Kostka}_r = \left\{ (\lambda, \mu) \in \mathbb{R}^{2r} \ : \ 
\begin{matrix}
\lambda_1 \geq \lambda_2 \geq \ldots \geq \lambda_r \geq 0 \\
\mu_1 \geq \mu_2 \geq \ldots \geq \mu_r \geq 0 \\
\sum_{i=1}^t \lambda_i \geq \sum_{i=1}^t \mu_i, \text{ for } 1 \leq t \leq r-1 \\
\sum_{i=1}^r \lambda_i = \sum_{i=1}^r \mu_i
\end{matrix}
 \right\} \subset \mathbb{R}^{2r}.\]
 We are interested in its lattice points, namely
 \[\sf{Kostka}_r^{\mathbb{Z}}:=\sf{Kostka}_r \cap \mathbb{Z}^{2r}.\]

Following \cite{hilbasis}, $(\lambda, \mu)\in \sf{Kostka}_r^{\mathbb{Z}}$ is \emph{reducible} if nontrivial $(\lambda^{\bullet}, \mu^{\bullet}), (\lambda^{\circ}, \mu^{\circ}) \in \sf{Kostka}_r^{\mathbb{Z}}$ exist such that
\begin{equation}
\label{eqn:thedecomposition}
 (\lambda, \mu) = (\lambda^{\bullet}, \mu^{\bullet}) + (\lambda^{\circ}, \mu^{\circ}). 
 \end{equation}
Also, (\ref{eqn:thedecomposition})
is \emph{commonly reducible} if there are common columns of $\lambda$ and $\mu$ giving $\lambda^{\bullet}$ and $\mu^{\bullet}$. 

By \cite[Proposition~5.7]{hilbasis}, Theorem~\ref{conj1} implies a
 strengthened version of \cite[Theorem~1.4]{hilbasis}; that is, we have
proved \cite[Conjecture~5.5]{hilbasis}:
  
\begin{cor} \label{conj2}
If $\lambda_1 > r$ then $(\lambda, \mu)$ is commonly reducible.
\end{cor}

\ytableausetup
 {mathmode, boxsize=.9em}
 
\begin{ex}
Consider the following decomposition ($\lambda_1 = 5, \ r = 3$):
\[ (\lambda, \mu) = \left( \ydiagram{0+5,0+3,0+1}, \ydiagram{0+3,0+3,0+2,0+1} \ \right) = 
 \left( \ydiagram{0+3,0+2,0+1}, \ydiagram{0+2,0+2,0+1,0+1} \ \right) + \left( \ydiagram{0+2,0+1}, \ydiagram{0+1,0+1,0+1} \ \right) = (\lambda^{\bullet}, \mu^{\bullet}) + (\lambda^{\circ}, \mu^{\circ}) \]
That $(\lambda,\mu)$ is commonly reducible holds by choosing columns $\{1,3,5\}$ of $(\lambda, \mu)$ to obtain $(\lambda^{\bullet}, \mu^{\bullet})$. For $(\lambda^{\circ}, \mu^{\circ})$, we can choose complementary columns $\{2,4\}$. \qed
\end{ex}

In Section~\ref{sec:2}, we prove Theorem~\ref{conj1}. In Section~\ref{sec:3} we extend Theorem~\ref{conj1} to show reducibility in 
the case that ${\sf cost}(\vec x)={\sf width}(\vec x)$ and $y>1$ (reducibility does not hold if $y=1$). We also determine all the exceptions in the case $y=1$.

\section{Proof of Theorem \ref{conj1}}\label{sec:2}

Let $[t]:= \{1,2, \ldots , t\}$. Interpret a list $\vec{y} = (y_1, \ldots , y_t)$ of nonzero integers as a map $y:[t] \rightarrow \mathbb{Z}\setminus \{0 \}$ given by $i \mapsto y_i$. Conversely, given a map $z : [t] \rightarrow \mathbb{Z}\setminus \{0 \}$, define a corresponding list of nonzero integers $\vec{z}$ by
\[ \vec{z} = (z(1), z(2), \ldots , z(t)) .\]
Define $x:[t] \rightarrow \mathbb{Z} \setminus \{0 \}$ to be \emph{generalized Catalan} if 
\[ \vec{x} = (x(1), \ldots , x(t))\]
is generalized Catalan. Let $\sf{Catalan}_t$ be the set of all such maps.

Let ${\mathfrak S}_t$ be the symmetric group of bijections of $[t]$ and let $\vec{x}=(x_1, \ldots , x_t)$ be a generalized Catalan list (hence $x \in \sf{Catalan_t}$). 

Construct $\pi \in {\mathfrak S}_t$ associated to $x$. 
Let $\pi(1) = 1$. For $2\leq q\leq t$, let
\begin{equation}
\label{eqn:sdef}
s := \min(i\in [t] : i \not\in \{\pi(1), \ldots , \pi(q-1) \} \text{\  and  $x(i)<0$})
\end{equation}
and 
\begin{equation}
\label{eqn:s'def}
s' := \min(i\in [t] : i \not\in \{\pi(1), \ldots , \pi(q-1) \} \text{\ and $x(i) > 0$}).
\end{equation}
Now
\begin{equation}
\label{eqn:pidefcases}
\pi(q)=\begin{cases}
s & \text{if $\left( \sum_{i=1}^{q-1} x(\pi(i)) \right) + x(s) \geq 0$ and $s<\infty$}\\
s' & \text{otherwise}
\end{cases}
\end{equation}

\begin{lem}
The construction of $\pi\in {\mathfrak S}_n$ is well-defined.
\end{lem}
\begin{proof}
Suppose that for some $2\leq q\leq t$, we are not in the first case of (\ref{eqn:pidefcases}), but
\[\{i : i \not\in \{\pi(1), \ldots , \pi(q-1) \} \text{\ and $x(i) >0$} \} = \emptyset.\] 
In that case,
$$ \left( \sum_{i=1}^{q-1} x(\pi(i)) \right) + x(s) \geq \left(\sum_{i=1}^{q-1} x(\pi(i)) \right) + \sum_{i \not\in \{\pi(1), \ldots , \pi(q-1)\} } x(i)  = 0$$
which means we used the first case, after all, a contradiction. 
\end{proof}

\begin{lem}
Let $x\circ \pi:[t]\to {\mathbb Z}\setminus \{0\}$ be function composition.
Then $ x \circ \pi \in {\sf Catalan}_t.$
\end{lem}
\begin{proof}
Since $\pi \in \mathfrak{S}_t$ and $x\in {\sf Catalan}_t$,
\[ \sum_{i=1}^t x \circ \pi (i) = \sum_{i=1}^t x(i) = 0 .\]
It remains to show that
\[\sum_{i=1}^q x \circ \pi(i) \geq 0 \]
for all $1 \leq q \leq t$. We proceed by induction on $q\geq 1$. 
In the base case, $q=1$, $x \circ \pi(1) = x(1) \geq 0$. Now suppose
$q \geq 2$, and let $s$ and $s'$ be as in \eqref{eqn:sdef} and \eqref{eqn:s'def}. 
If $\pi(q) = s$, then
\[ \sum_{i=1}^q x \circ \pi (i) = \left( \sum_{i=1}^{q-1} x(\pi(i)) \right)+ x(s) \geq 0 . \]
If $\pi(q) = s'$, then
\[ \sum_{i=1}^q x \circ \pi(i) = \left( \sum_{i=1}^{q-1} x \circ \pi (i) \right) + x(s') > \sum_{i=1}^{q-1} x \circ \pi(i) \geq 0 . \]
This completes the induction.
\end{proof}

\begin{ex} \label{ex2}
Consider the list $\vec{x}$ in Example \ref{ex1}:
\[ \vec{x} = (5,5,4,4,-3,-3,-3,-3,-3,-1,5,5,5,3,-4,-4,-4,-4,-4). \]
A permutation $\pi \in \mathfrak{S}_t$ associated to $x$ is
\[ \pi = (1 \ 5 \ 2 \ 6 \ 7 \ 3 \ 8 \ 4 \ 9 \ 10 \ 11 \ 15 \ 12 \ 16 \ 17 \ 13 \ 18 \ 14 \ 19) \]
where $\pi$ is written in one-line notation. The sequence that corresponds to $x \circ \pi$ is
\[ \vec{w} = (5,-3,5,-3,-3,4,-3,4,-3,-1,5,-4,5,-4,-4,5,-4,3,-4), \]
which is generalized Catalan. Hence $x \circ \pi \in {\sf Catalan}_t$. Below we colored the Dyck path in Figure \ref{fig1}. The color of an edge represents the run in which an edge is contained.
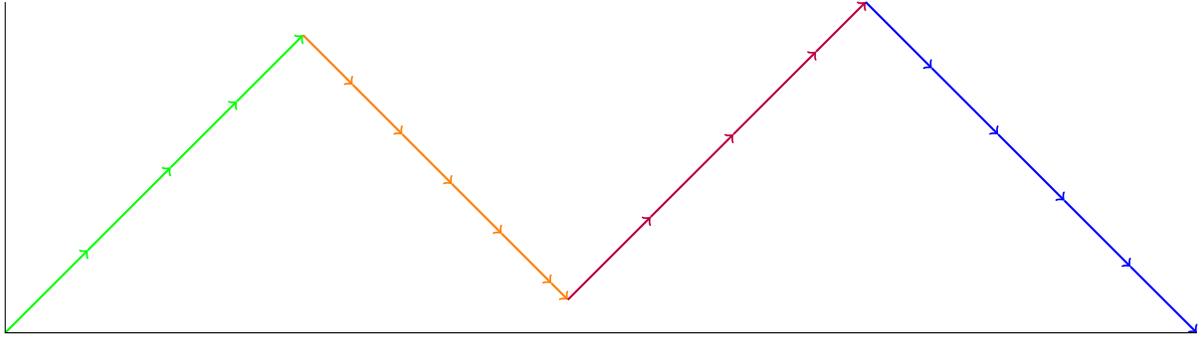
\begin{center}
\begin{tikzpicture}[scale = 0.22]
\draw [->, green, thick] (0,0) -- (5,5);
\draw [->, green, thick] (5,5) -- (10,10);
\draw [->, green, thick] (10,10) -- (14,14);
\draw [->, green, thick] (14,14) -- (18,18);
\draw [->, orange, thick] (18,18) -- (21,15);
\draw [->, orange, thick] (21,15) -- (24,12);
\draw [->, orange, thick] (24,12) -- (27,9);
\draw [->, orange, thick] (27,9) -- (30,6);
\draw [->, orange, thick] (30,6) -- (33,3);
\draw [->, orange, thick] (33,3) -- (34,2);
\draw [->, purple, thick] (34,2) -- (39,7);
\draw [->, purple, thick] (39,7) -- (44,12);
\draw [->, purple, thick] (44,12) -- (49,17);
\draw [->, purple, thick] (49,17) -- (52,20);
\draw [->, blue, thick] (52,20) -- (56,16);
\draw [->, blue, thick] (56,16) -- (60,12);
\draw [->, blue, thick] (60,12) -- (64,8);
\draw [->, blue, thick] (64,8) -- (68,4);
\draw [->, blue, thick] (68,4) -- (72,0);
\draw (0,20) -- (0,0) -- (72,0);
\end{tikzpicture}
\end{center}
For each $w_i = x \circ \pi (i)$, there is a corresponding entry $x_{\pi(i)}$ of $\vec{x}$. Below is the visualization of $\vec{w}$ with the coloring inherited from $\vec{x}$.
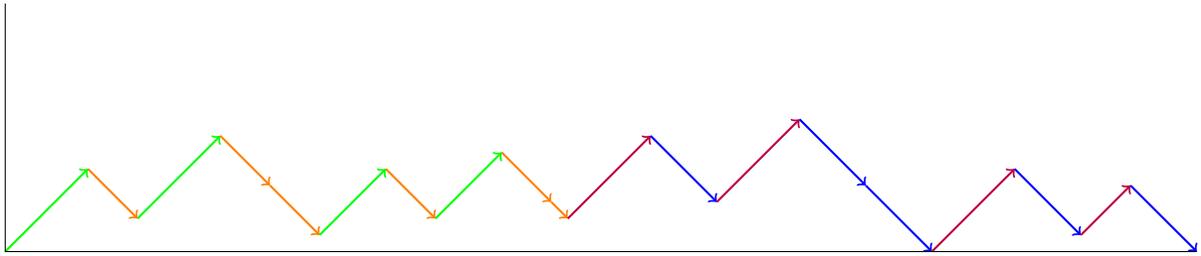
\begin{center}
\begin{tikzpicture}[scale = 0.22]
\draw [->, green, thick] (0,0) -- (5,5);
\draw [->, orange, thick] (5,5) -- (8,2);
\draw [->, green, thick] (8,2) -- (13,7);
\draw [->, orange, thick] (13,7) -- (16,4);
\draw [->, orange, thick] (16,4) -- (19,1);
\draw [->, green, thick] (19,1) -- (23,5);
\draw [->, orange, thick] (23,5) -- (26,2);
\draw [->, green, thick] (26,2) -- (30,6);
\draw [->, orange, thick] (30,6) -- (33,3);
\draw [->, orange, thick] (33,3) -- (34,2);
\draw [->, purple, thick] (34,2) -- (39,7);
\draw [->, blue, thick] (39,7) -- (43,3);
\draw [->, purple, thick] (43,3) -- (48,8);
\draw [->, blue, thick] (48,8) -- (52,4);
\draw [->, blue, thick] (52,4) -- (56,0);
\draw [->, purple, thick] (56,0) -- (61,5);
\draw [->, blue, thick] (61,5) -- (65,1);
\draw [->, purple, thick] (65,1) -- (68,4);
\draw [->, blue, thick] (68,4) -- (72,0);
\draw (0,15) -- (0,0) -- (72,0);
\end{tikzpicture}
\end{center}
In this Dyck path, any green edge appears earlier than any purple edge, and any orange edge appears earlier than any blue edge. \qed

\end{ex}

\begin{lem} \label{prop1}
Let $1\leq i<j\leq t$. 
\begin{itemize}
    \item[(I)] If $x(i), x(j) >0$, then $\pi^{-1}(i)<\pi^{-1}(j)$.
    \item[(II)] If $x(i), x(j) <0$, then $\pi^{-1}(i)<\pi^{-1}(j)$.
    \item[(III)] If $x(i)<0$ and $x(j)>0$, then $\pi^{-1}(i)<\pi^{-1}(j)$.
\end{itemize}
\end{lem}
\begin{proof}
If $\pi^{-1}(j) = 1$ then $\pi(1) = 1 = j$, a contradiction. Hence
$\pi^{-1}(j) >1$. Now we consider step $\pi^{-1}(j)$ of the construction for each of the cases.

(I): Here,
\[j=\min(l:l \not\in \{\pi(1), \ldots , \pi(\pi^{-1}(j)-1) \} \text{ and } x(l) >0)\] so 
\[i \in \{\pi(1), \ldots , \pi(\pi^{-1}(j)-1) \}\] 
which means that $\pi^{-1}(i) < \pi^{-1}(j)$. 

(II): Now,
\[j = \min(l : l \not\in \{\pi(1), \ldots , \pi(\pi^{-1}(j)-1) \} \text{ and } x(l) <0 )\] 
so 
\[i \in \{\pi(1), \ldots , \pi(\pi^{-1}(j)-1) \}\] 
implying $\pi^{-1}(i) < \pi^{-1}(j)$. 

(III): In this case, 
\[s = \min(l: l \not\in \{\pi(1), \ldots , \pi(\pi^{-1}(j)-1) \} \text{ and } x(l)<0).\] 
Since 
\[x(\pi(\pi^{-1}(j))) = x(j) >0,\] 
we have that 
\[\{l : l \not\in \{ \pi(1), \ldots , \pi(\pi^{-1}(j)-1) \} \text{ and } x(l)<0 \} = \emptyset\] 
or
\begin{equation}
\label{eqn:myoneeqn}
 \left( \sum_{k=1}^{\pi^{-1}(j)-1} x(\pi(k)) \right) + x(s) <0.
 \end{equation}
If 
\[\{l : l \not\in \{\pi(1), \ldots , \pi(\pi^{-1}(j)-1) \} \text{ and } x(l)<0 \} = \emptyset,\] 
we see that
\[i \in \{\pi(1), \ldots , \pi(\pi^{-1}(j)-1) \}\]
which implies $\pi^{-1}(i) < \pi^{-1}(j)$ as desired. 

Suppose (\ref{eqn:myoneeqn}) holds. By \eqref{eqn:sdef},
$$\{l \in \{\pi(1), \ldots , \pi(\pi^{-1}(j)-1) \} : x(l) <0 \} = [s-1]\cap \{h \ : \ x(h)<0 \}.$$
By \eqref{eqn:s'def},
$$\{l \in \{\pi(1), \ldots , \pi(\pi^{-1}(j)-1) \} : x(l) >0 \} = [j-1] \cap \{h \ : \ x(h)>0 \}.$$
If $s < j$,
$$\left( \sum_{k=1}^{\pi^{-1}(j)-1} x(\pi(k)) \right) + x(s) = \sum_{h=1}^s x(h) + \sum_{k \in \{s+1, \ldots , j-1 \} \cap \{h : x(h) >0 \}} x(k) \geq \sum_{h=1}^s x(h) \geq 0,$$
a contradiction. Thus $s \geq j (>i)$. This means 
\[i \in \{\pi(1), \ldots , \pi(\pi^{-1}(j)-1) \},\] 
by \eqref{eqn:sdef}. Therefore, $\pi^{-1}(i) < \pi^{-1}(j)$, as desired.  
\end{proof}

Let $T=\{i_1, \ldots , i_a \} \subset [t]$, where $i_1 < i_2 < \ldots < i_a$. A map $w:T \rightarrow \mathbb{Z} \setminus \{0 \}$ is \emph{generalized Catalan} if $(w(i_1), \ldots , w(i_a))$ is a generalized Catalan list. Let ${\sf Catalan}_T$ be the set of such maps (hence,
in particular ${\sf Catalan}_{[t]}={\sf Catalan}_t$).

\begin{prop} \label{lem1}
Let $T \subset [t]$. If 
\[(x \circ \pi)|_T : T \rightarrow \mathbb{Z} \setminus \{0 \}\in {\sf Catalan}_T\] 
then 
\[x|_{\pi(T)} : \pi(T) \rightarrow \mathbb{Z} \setminus \{0  \}\in {\sf Catalan}_{\pi(T)}.\]
\end{prop}
\begin{proof}
Observe that $x|_{\pi(T)} : \pi(T) \rightarrow \mathbb{Z} \setminus \{0 \}$ maps $\pi(i) \mapsto x(\pi(i))$. For each $k \in T$, consider
$$ {\mathcal S}:=\sum_{h\in T : \pi(h) < \pi(k)} x(\pi(h)).$$
By definition of ${\sf Catalan}_{\pi(T)}$, it suffices to show ${\mathcal S}\geq 0$ for all $k \in T$.

Fix $k \in T$. Suppose $x(\pi(k)) >0$. For $h\in T$ such that $x(\pi(h)) >0$, by Lemma \ref{prop1}(I), 
\[\pi(h) < \pi(k)\iff h<k.\] 
This means that
$$ \sum_{h \in T : \pi(h) < \pi(k), x(\pi(h))>0} x(\pi(h)) = \sum_{h \in T : h<k, x(\pi(h))>0} x(\pi(h)).$$
For $h \in T$ such that $x(\pi(h))<0$, by Lemma \ref{prop1}(III), we have that 
\[ \pi(h)<\pi(k) \Rightarrow h<k;\] 
that is,
$$ \sum_{h \in T: \pi(h)<\pi(k), x(\pi(h))<0} x(\pi(h)) \geq \sum_{h \in T : h<k, x(\pi(h))<0} x(\pi(h)).$$
Therefore, 
$$ \sum_{h \in T : \pi(h) < \pi(k)} x(\pi(h)) \geq \sum_{h \in T : h<k} x(\pi(h)) \geq 0,$$
as desired. The last inequality holds because $(x \circ \pi)|_T \in {\sf Catalan}_T$.

Suppose $x(\pi(k))<0$. For $h \in T$ such that $x(\pi(h))<0$, by Lemma \ref{prop1}(II), 
\[\pi(h) < \pi(k)  \iff  h<k.\] 
This means that
$$ \sum_{h \in T : \pi(h) < \pi(k), x(\pi(h))<0} x(\pi(h)) = \sum_{h \in T : h<k, x(\pi(h)) <0} x(\pi(h)).$$
For $h \in T$ such that $x(\pi(h)) >0$, by Lemma \ref{prop1}(III), 
\[ \pi(k)<\pi(h) \Rightarrow k<h. \]
Equivalently, for $h \in T$ such that $x(\pi(h))>0$,
\[ h<k \Rightarrow \pi(h) < \pi(k). \]
This implies that
$$ \sum_{h \in T: \pi(h) < \pi(k), x(\pi(h))>0} x(\pi(h)) \geq \sum_{h \in T : h<k, x(\pi(h))>0} x(\pi(h)).$$
Hence,
$$ \sum_{h \in T : \pi(h) < \pi(k)} x(\pi(h)) \geq \sum_{h \in T : h<k} x(\pi(h)) \geq 0$$
as desired. The last inequality holds because $(x \circ \pi)|_T\in {\sf Catalan}_T$.
\end{proof}

\begin{ex}
Let $\vec{x}$ and $\vec{w}$ be as in Example \ref{ex2}, and let 
\[ T = \{ 3,4,6,9,10,11,12,13,14,15 \} \subset [19]. \]
Then $(x \circ \pi)|_T \in {\sf Catalan}_T$ because the corresponding list 
\[ \vec{w}^{\bullet}=(5,-3,4,-3,-1,5,-4,5,-4,-4) \]
is Catalan. Below is the visualization of $\vec{w}^{\bullet}$ as a sublist of $\vec{w}$, where red edges represent the entries of $\vec{w}^{\bullet}$.
\begin{center}
\begin{tikzpicture}[scale = 0.22]
\draw [->, gray] (0,0) -- (5,5);
\draw [->, gray] (5,5) -- (8,2);
\draw [->, red, thick] (8,2) -- (13,7);
\draw [->, red, thick] (13,7) -- (16,4);
\draw [->, gray] (16,4) -- (19,1);
\draw [->, red, thick] (19,1) -- (23,5);
\draw [->, gray] (23,5) -- (26,2);
\draw [->, gray] (26,2) -- (30,6);
\draw [->, red, thick] (30,6) -- (33,3);
\draw [->, red, thick] (33,3) -- (34,2);
\draw [->, red, thick] (34,2) -- (39,7);
\draw [->, red, thick] (39,7) -- (43,3);
\draw [->, red, thick] (43,3) -- (48,8);
\draw [->, red, thick] (48,8) -- (52,4);
\draw [->, red, thick] (52,4) -- (56,0);
\draw [->, gray] (56,0) -- (61,5);
\draw [->, gray] (61,5) -- (65,1);
\draw [->, gray] (65,1) -- (68,4);
\draw [->, gray] (68,4) -- (72,0);
\draw (0,15) -- (0,0) -- (72,0);
\end{tikzpicture}
\end{center}
Since
\[ \pi(T) = \{2,3,6,9,10,11,12,15,16,17 \}, \]
the list that corresponds to $x |_{\pi(T)}$ is
\[ \vec{x}^{\bullet} = (5,4,-3,-3,-1,5,5,-4,-4,-4) \]
and it is Catalan. Hence $x|_{\pi(T)} \in {\sf Catalan}_{\pi(T)}$. Below is the visualization of $\vec{x}^{\bullet}$ as a sublist of $\vec{x}$, where red edges represent the entries of $\vec{x}^{\bullet}$.
\begin{center}
\begin{tikzpicture}[scale = 0.22]
\draw [->, gray] (0,0) -- (5,5);
\draw [->, red, thick] (5,5) -- (10,10);
\draw [->, red, thick] (10,10) -- (14,14);
\draw [->, gray] (14,14) -- (18,18);
\draw [->, gray] (18,18) -- (21,15);
\draw [->, red, thick] (21,15) -- (24,12);
\draw [->, gray] (24,12) -- (27,9);
\draw [->, gray] (27,9) -- (30,6);
\draw [->, red, thick] (30,6) -- (33,3);
\draw [->, red, thick] (33,3) -- (34,2);
\draw [->, red, thick] (34,2) -- (39,7);
\draw [->, red, thick] (39,7) -- (44,12);
\draw [->, gray] (44,12) -- (49,17);
\draw [->, gray] (49,17) -- (52,20);
\draw [->, red, thick] (52,20) -- (56,16);
\draw [->, red, thick] (56,16) -- (60,12);
\draw [->, red, thick] (60,12) -- (64,8);
\draw [->, gray] (64,8) -- (68,4);
\draw [->, gray] (68,4) -- (72,0);
\draw (0,20) -- (0,0) -- (72,0);
\end{tikzpicture}
\end{center}
\qed
\end{ex}

\noindent
\emph{Proof of Theorem~\ref{conj1}:}
A run is an \emph{up-run} if it consists of positive integers. Similarly, a run is a \emph{down-run} if it consists of negative integers. Let $U_i \subset [t]$ be the set of indices of the elements in the $i$th up-run of $\vec{x}$, and let $D_j \subset [t]$ be the set of indices of the elements in the $j$th down-run of $\vec{x}$ $(1 \leq i, j \leq y)$. Let $\alpha_i$ be the maximum of the elements in the $i$th up-run of $\vec{x}$ and let $\beta_j$ be the maximum (in absolute value) of the elements in the $j$th down-run of $\vec{x}$.

Let 
\[\gamma_i := \min(l : \pi(l) \in U_i) \text{\ and $\delta_j := \max(l : \pi(l) \in D_j)$.}\] 
The \emph{$i$-th up-phase} $\Phi_i \subset [t]$ is 
\[ \Phi_i := \{\gamma_i, \gamma_i+1, \ldots , \gamma_{i+1}-1 \} \text{\ if $i<y$.} \]
The \emph{$j$-th down-phase} $\Psi_j \subset [t] \cup \{0 \}$ is 
\[ \Psi_j := \{\delta_{j-1}, \delta_{j-1}+1, \ldots , \delta_j - 1 \} \text{\ if $j>1$.}\]
Also, we set $\Phi_y = \{\gamma_{y}, \ldots , t\}$ and $\Psi_1 = \{0, 1, 2, \ldots , \delta_1 - 1\}$.
Notice that $\{\Phi_i:1\leq i\leq y\}$ is a set-partition of $[t]$. Also, 
$\{\Psi_j: 1\leq j\leq y\}$ is a set-partition of $[t-1] \cup \{0 \}$.  

\begin{prop} \label{obs1}
 If $h \in \Phi_i$ and $x(\pi(h))<0$ then $\sum_{k=1}^h x(\pi(k)) \in [0, \alpha_i)$.
\end{prop}
\begin{proof}
 Suppose $\sum_{k=1}^h x(\pi(k)) \geq \alpha_i$; assume $h$ is minimal with this property.
Then we have that $x(\pi(h-1)) > 0$. Since 
\[ \pi(h) = \min(l : l \not\in \{\pi(1), \ldots , \pi(h-1) \} \text{ and } x(l) <0 \}\]
 by definition of $\pi(h)$, we also have 
\[ \pi(h) = \min(l : l \not\in \{\pi(1), \ldots , \pi(h-2) \} \text{ and } x(l)<0 \}. \]
However, this means that
$$ \left( \sum_{k=1}^{h-2} x(\pi(k)) \right) + x(\pi(h)) \geq \alpha_i - x(\pi(h-1)) \geq 0$$
which implies that we should have chosen the negative element $x(\pi(h))$ instead of $x(\pi(h-1))$ at step $h-1$. This contradicts the construction of $\pi$.
\end{proof}

\begin{prop} \label{obs2}
If $h \in \Psi_j$ and $x(\pi(h+1)) >0$ then $\sum_{k=1}^h x(\pi(k)) \in [0, \beta_j)$.
\end{prop}
\begin{proof}
Since $h \in \Psi_j$, we get 
\[ s:= \min(l : l \not\in \{\pi(1), \ldots , \pi(h) \} \text{ and } x(l) <0 ) \in D_j. \]
Since $x(\pi(h+1)) >0$, we have that
$ \left( \sum_{k=1}^h x(\pi(k)) \right) + x(s) < 0$ which means that
$$ \sum_{k=1}^h x(\pi(k)) < -x(s) = |x(s)| \leq \beta_j$$
as desired. Thus $\sum_{k=1}^h x(\pi(k)) \in [0, \beta_j)$.
\end{proof}

Let $u_i := | \{h \in \Phi_i : x(\pi(h))<0 \}|$ and $d_j := | \{h \in \Psi_j : x(\pi(h+1)) >0 \}|$. Since
$\{\Phi_i:1\leq i\leq y\}$ is a set-partition of $[t]$, 
\[ \sum_{i=1}^y u_i = |\{h \in [t] : x(\pi(h))<0 \}|.\]
Similarly, since $\{\Psi_j: 1\leq j\leq y\}$ is a set-partition of $[t-1]\cup \{0\}$,
\[ \sum_{j=1}^y d_j = | \{h \in [t-1] \cup \{0 \} : x(\pi(h+1)) >0 \} | = | \{h \in [t] : x(\pi(h)) >0 \} |. \]
The previous two sentences show that
\begin{equation} \label{eqn:costwidth}
\sum_{i=1}^y u_i + \sum_{j=1}^y d_j = t = \width(\vec{x}) > \cost(\vec{x}) = \sum_{i=1}^y \alpha_i + \sum_{j=1}^y \beta_j.
\end{equation}
Hence, there exists $i$ such that $u_i > \alpha_i$ or there exists $j$ such that $d_j > \beta_j$.
It remains to analyze these two cases.

\noindent \emph{Case A} (There exists $i$ such that $u_i > \alpha_i$):

For each $h \in \Phi_i$ such that $x(\pi(h))<0$, we have 
\[ \sum_{k=1}^h x(\pi(k)) \in [0, \alpha_i)\]
due to Proposition \ref{obs1}. Since 
\[ u_i = | \{h \in \Phi_i : x(\pi(h)) <0 \}| > \alpha_i , \]
by pigeonhole, there exists $h_1, h_2 \in \Phi_i$ such that $h_1<h_2$, 
\[ x(\pi(h_1)), x(\pi(h_2))<0, \]
and
$$ \sum_{k=1}^{h_1} x(\pi(k)) = \sum_{k=1}^{h_2} x(\pi(k)).$$

Recall that $ \gamma_i = \min(l : \pi(l) \in U_i)$ and  $\gamma_i = \min(\Phi_i)$. 
Suppose $k \in \{h_1+1, \ldots , h_2 \}$ satisfies $x(\pi(k))<0$. Since 
\[ \gamma_i < k, \ x(\pi(\gamma_i)) >0, \text{ and } x(\pi(k)) <0, \] by Lemma \ref{prop1} (III), we get $\pi(\gamma_i) < \pi(k)$. Since $\pi(\gamma_i) \in U_i$, this implies that $\max(U_i) < \pi(k)$.
On the other hand, suppose $l \in \{h_1 +1, \ldots , h_2 \}$ satisfies $x(\pi(l))>0$. Since $l \in \Phi_i$, we have that $\pi(l) \in U_i$, which means that $\pi(l) \leq \max(U_i)$. 

Therefore, for any $k, l \in \{h_1+1, \ldots, h_2 \}$ such that $x(\pi(k))<0$ and $x(\pi(l))>0$, 
\begin{equation} \label{eqn:posneg}
\pi(l) \leq \max(U_i) < \pi(k).
\end{equation}
Let $T_1 = \{h_1+1, \ldots , h_2 \}$ and $T_2 = [t] \setminus T_1 = \{1, \ldots , h_1 \} \cup \{h_2+1, \ldots , t\}$. 

Consider the sublist of $\vec{x}$ that corresponds to the map $x|_{\pi(T_1)} : \pi(T_1) \rightarrow \mathbb{Z} \setminus \{0\}$. In this sublist, by (\ref{eqn:posneg}), any positive element appears earlier than any negative element. Therefore, this sublist consists of one up-run and one down-run, which is generalized Catalan.

Next, consider the map $x|_{\pi(T_2)} : \pi(T_2) \rightarrow \mathbb{Z} \setminus \{0 \}$. Since $x \circ \pi\in {\sf Catalan}_t$, 
\[ \sum_{k=1}^r x(\pi(k)) \geq 0 \]
 for all $r \leq h_1$ and 
\[ \sum_{k=1}^{h_1} x(\pi(k)) + \sum_{k=h_2+1}^r x(\pi(k)) = \sum_{k=1}^r x(\pi(k)) \geq 0 \]
for all $r \geq h_2+1$ which means that $(x \circ \pi)|_{T_2}\in {\sf Catalan}_{T_2}$. Thus by Proposition \ref{lem1}, $x|_{\pi(T_2)}\in {\sf Catalan}_{\pi(T_2)}$.

Since $T_1 \cup T_2 = [t]$, we have that $\pi(T_1) \cup \pi(T_2) = [t]$. Therefore, two sublists of $\vec{x}$ that correspond to the maps $x|_{\pi(T_1)}$ and $x|_{\pi(T_2)}$ are complement to each other and are both generalized Catalan. Hence $\vec{x}$ is reducible.

\noindent \emph{Case B} (There exists $j$ such that $d_j > \beta_j$):

For each $h \in \Psi_j$ such that $x(\pi(h+1))>0$,  
\[ \sum_{k=1}^h x(\pi(k)) \in [0, \beta_j),\] 
by Proposition \ref{obs2}. Since 
\[ d_j = |\{h \in \Psi_j : x(\pi(h+1)) >0 \}| > \beta_j,\]
 by pigeonhole, there exists $h_1, h_2 \in \Psi_j$ such that $h_1<h_2$, 
\[ x(\pi(h_1+1)), x(\pi(h_2+1)) >0, \]
and
$$ \sum_{k=1}^{h_1} x(\pi(k)) = \sum_{k=1}^{h_2} x(\pi(k)).$$

Recall that $\delta_j = \max(l : \pi(l) \in D_j)$ and $\max(\Psi_j) = \delta_j-1$. 
Suppose $k \in \{h_1+1, \ldots , h_2 \}$ satisfies $x(\pi(k)) >0$. Since 
\[ k<\delta_j, \ x(\pi(\delta_j))<0, \text{ and } x(\pi(k))>0, \]
by Lemma \ref{prop1} (III), we have $\pi(\delta_j) > \pi(k)$. Since $\pi(\delta_j) \in D_j$, this implies that $\min(D_j) > \pi(k)$. On the other hand, suppose $l \in \{h_1+1, \ldots , h_2 \}$ satisfies $x(\pi(l)) <0$. Since $l \in \Psi_j$ and $l \geq h_1+1 > \delta_{j-1}$, we have that $\pi(l) \in D_j$, which means that $\pi(l) \geq \min(D_j)$.

Therefore, for any $k, l \in \{h_1+1, \ldots , h_2 \}$ such that $x(\pi(k))>0$ and $x(\pi(l))<0$, 
\begin{equation} \label{eqn:posneg2}
\pi(l) \geq \min(D_j) > \pi(k) .
\end{equation}

Let $T_1 = \{h_1+1, \ldots , h_2 \}$ and $T_2 = [t] \setminus T_1 = \{1, \ldots , h_1 \} \cup \{h_2+1, \ldots , t\}$. 

Consider the sublist of $\vec{x}$ that corresponds to the map $x|_{\pi(T_1)} : \pi(T_1) \rightarrow \mathbb{Z} \setminus \{0 \}$.  In this sublist, by (\ref{eqn:posneg2}), any positive element appears earlier than any negative element. Thus, this sublist consists of one up-run and one down-run, and therefore it is generalized Catalan.

Next, consider the map $x|_{\pi(T_2)} : \pi(T_2) \rightarrow \mathbb{Z} \setminus \{0 \}$. Since $x \circ \pi\in {\sf Catalan}_t$, 
\[ \sum_{k=1}^r x(\pi(k)) \geq 0 \] for all $r \leq h_1$ and 
\[ \sum_{k=1}^{h_1} x(\pi(k)) + \sum_{k=h_2+1}^r x(\pi(k)) = \sum_{k=1}^r x(\pi(k)) \geq 0 \]
 for all $r \geq h_2+1$ which means that $(x \circ \pi)_{T_2}\in {\sf Catalan}_{T_2}$. Thus by Proposition \ref{lem1}, $x|_{\pi(T_2)}\in {\sf Catalan}_{\pi(T_2)}$. 

Since $T_1 \cup T_2 = [t]$, $\pi(T_1) \cup \pi(T_2) = [t]$. Therefore, two sublists of $\vec{x}$ that correspond to the maps $x|_{\pi(T_1)}$ and $x|_{\pi(T_2)}$ are complementary and are both generalized Catalan. Hence $\vec{x}$ is reducible. \qed

\section{The case ${\sf cost}(\vec x)={\sf width}(\vec x)$}\label{sec:3}
\begin{thm}\label{thm:equality}
If $\cost(\vec{x}) = \width(\vec{x})$ and $y>1$ then $\vec{x}$ is reducible.
\end{thm}

The hypothesis $y>1$ cannot be dispensed with. For example, $\vec x=(2,-1,-1)$ is not reducible. 
\begin{proof}
The proof is the same as that of Theorem \ref{conj1}, until (\ref{eqn:costwidth}), which we replace with 
\[ \sum_{i=1}^y u_i + \sum_{j=1}^y d_j = t = \width(\vec{x}) = \cost(\vec{x}) = \sum_{i=1}^y \alpha_i + \sum_{j=1}^y \beta_j. \]
If there exists some $i$ such that $u_i > \alpha_i$ or if there exists some $j$ such that $d_j>\beta_j$, then we 
are done by the arguments of \emph{Case A} and \emph{Case B} in the proof of Theorem \ref{conj1}. Hence our only concern is when $u_i=\alpha_i$ for all $i$ and $d_j=\beta_j$ for all $j$.

Suppose there exists $h \in \Phi_1$ such that
\[ \sum_{k=1}^h x(\pi(k)) = 0. \]
Since $h \in \Phi_1=\{1, 2, \ldots , \gamma_2-1 \}$, $h \neq 0$. Since $y>1$, $h \neq t$. Let $S_1 = \{1, 2, \ldots , h\}$ and $S_2 = \{h+1, h+2, \ldots , t\}$. Then $(x \circ \pi)|_{S_1} \in {\sf Catalan}_{S_1}$ and $(x \circ \pi)|_{S_2} \in {\sf Catalan}_{S_2}$. Thus by Proposition \ref{lem1}, $x|_{\pi(S_1)} \in {\sf Catalan}_{\pi(S_1)}$ and $x|_{\pi(S_2)} \in {\sf Catalan}_{\pi(S_2)}$. Since $S_1 \cup S_2 = [t]$, $\pi(S_1) \cup \pi(S_2) = [t]$. Therefore, two sublists of $\vec{x}$ that correspond to the maps $x|_{\pi(S_1)}$ and $x|_{\pi(S_2)}$ are complementary and are both generalized Catalan. Hence $\vec{x}$ is reducible.

Suppose otherwise, that there is no $h \in \Phi_1$ such that
\[ \sum_{k=1}^h x(\pi(k)) = 0. \]
For each $h \in \Phi_1$ such that $x(\pi(h))<0$, we have
\[ \sum_{k=1}^h x(\pi(k)) \in (0, \alpha_1) \]
due to Proposition \ref{obs1}. Since
\[ u_1 = |\{h \in \Phi_1 \ : \ x(\pi(h))<0 \}| > \alpha_1 -1, \]
by pigeonhole, there exists $h_1, h_2 \in \Phi_1$ such that $h_1<h_2$,
\[ x(\pi(h_1)), x(\pi(h_2)) <0, \]
and
\[ \sum_{k=1}^{h_1} x(\pi(k)) = \sum_{k=1}^{h_2} x(\pi(k)). \]
This is exactly what we had in \emph{Case A} (of the proof of Theorem~\ref{conj1}. Hence 
continuing the argument there, we obtain our claim.
\end{proof}
\begin{thm} \label{thm:y1}
If $\cost(\vec{x}) = \width(\vec{x})$, $y=1$, and $\vec{x}$ is not reducible, then all the entries of $\vec{x}$ are either $\alpha_1$ or $-\beta_1$. Furthermore, $\alpha_1$ and $\beta_1$ are relatively prime. 
\end{thm}
\begin{proof}
Since $y=1$, there is one up-run and one down-run in $\vec{x}$. Hence changing the order inside the up-run or inside the down-run doesn't affect the reducibility of $\vec{x}$. Thus we can assume that $x_1$ is the minimum of any $x_i$ in the up-run.

Suppose $x_1 < \alpha_1$. Construct $\sigma \in \mathfrak{S}_t$ associated to $x$. Let $\sigma(1) = 1$. For $2 \leq q \leq t$, let
\[ s_q:= \min(i \in [t] \ : \ i \not\in \{\sigma(1), \ldots , \sigma(q-1) \} \text{ and } x(i)<0) \]
and
\[ s_q':= \min(i \in [t] \ : \ i \not\in \{\sigma(1), \ldots , \sigma(q-1) \} \text{ and } x(i)>0). \]
Now
\[ \sigma(q) = \begin{cases} s_q & \text{if $\sum_{i=1}^{q-1} x(\sigma(i)) \geq 0$} \\ s_q' & \text{otherwise} \end{cases} \]
\begin{lem}
The construction of $\sigma \in \mathfrak{S}_t$ is well-defined.
\end{lem}
\begin{proof}
If $\sum_{i=1}^{q-1}x(\sigma(i)) \geq0$, then 
\[ \sum_{j \not\in \{ \sigma(1), \ldots , \sigma(q-1) \}} x(j) = \sum_{j=1}^t x(j) - \sum_{i=1}^{q-1}x(\sigma(i)) \leq 0. \]
Thus $s_q <\infty$. If $\sum_{i=1}^{q-1}x(\sigma(i)) < 0$, then
\[ \sum_{j \not\in \{ \sigma(1), \ldots , \sigma(q-1) \}} x(j) = \sum_{j=1}^t x(j) - \sum_{i=1}^{q-1} x(\sigma(i)) >0 .\]
Hence $s_{q'} < \infty$.
\end{proof}
Let $M_q := \sum_{i=1}^q x(\sigma(i))$. Suppose $M_q = 0$ for some $1 \leq q \leq t-1$. Then the sublist of $\vec{x}$ consisting of $x_{\sigma(1)}, \ldots , x_{\sigma(q)}$ is generalized Catalan. This sublist and its complement witness the reducibility of $\vec{x}$, a contradiction. Hence $M_q \neq 0$ for all $1 \leq q \leq t-1$.

\begin{lem}
$1- \beta_1 \leq M_q \leq \alpha_1 -1$ for all $1 \leq q \leq t$.
\end{lem}
\begin{proof}
We prove by induction on $q \geq1$. In the base case, $q=1$, $M_q = x(\sigma(1)) = x_1$ so $1-\beta_1 \leq M_1 \leq \alpha_1-1$. Now suppose $q \geq 2$. If $\sigma(q) = s_q$, then
\[ M_q = M_{q-1} + x(s_q) \geq 1 + x(s_q) \geq 1-\beta_1 \]
and
\[ M_q = M_{q-1}+x(s_q) \leq M_{q-1} \leq \alpha_1-1.\]
If $\sigma(q) = s_q'$, then
\[ M_q = M_{q-1} + x(s_q') \leq -1 + x(s_q') \leq \alpha_1-1\]
and
\[ M_q = M_{q-1} + x(s_q') \geq M_{q-1} \geq 1- \beta_1.\]
This completes the induction.
\end{proof}
Therefore, for $1 \leq q \leq t-1$, $M_q \in V = \{1-\beta_1, 2-\beta_1, \ldots , \alpha_1-1 \} - \{0 \}$. Since
\[ |V| = \alpha_1+\beta_1-2 = \cost(\vec{x}) -2 = \width(\vec{x})-2 = t-2 <t-1, \]
by pigeonhole, there exists $q_1, q_2 \in [t-1]$ such that $q_1<q_2$ and $M_{q_1} = M_{q_2}$. Then
\[ \sum_{i=q_1+1}^{q_2} x(\sigma(i)) = 0 .\]
Consider the sublist of $\vec{x}$ consists of $x_{\sigma(q_1+1)}, x_{\sigma(q_1+2)}, \ldots , x_{\sigma(q_2)}$. This sublist and its complement witness the reducibility of $\vec{x}$, a contradiction.

This shows that $x_1 = \alpha_1$, which means that all the entries in the up-run are $\alpha_1$. Similarly, all the entries in the down-run are $-\beta_1$. (Indeed, if we consider the sequence $\vec{w} = (-x_t, -x_{t-1}, \ldots , -x_1)$, the down-run of $\vec{x}$ corresponds to the up-run of $\vec{w}$)

Finally, if $\alpha_1$ and $\beta_1$ are not relatively prime, then choose $\beta_1/gcd(\alpha_1, \beta_1)$ many positive entries and $\alpha_1/gcd(\alpha_1, \beta_1)$ many negative entries to obtain a sublist $\vec{x}^{\bullet}$. $\vec{x}^{\bullet}$ and its complement sublist $\vec{x}^{\circ}$ witness the reduciblility of $\vec{x}$, a contradiction. Hence $\alpha_1$ and $\beta_1$ are relatively prime.
\end{proof}

Given a partition $\lambda$, the \textit{conjugate} $\lambda'$ is the partition whose Young diagram is the transpose of the Young diagram of $\lambda$. For the similar reason (\cite[Proposition~5.7]{hilbasis}) that Theorem~\ref{conj1} implies
Corollary~\ref{conj2}, Theorem~\ref{conj1}, Theorem~\ref{thm:equality}, and Theorem \ref{thm:y1} implies:

\begin{cor}
If $\lambda_1 = r$ and $(\lambda, \mu)$ is not commonly reducible, then $\lambda$ and $\mu$ are both rectangles. Furthermore, $\lambda_1$ and $\mu_1$ are relatively prime.
\end{cor}
\begin{proof}
Define a sequence $\vec{x}$ of length $\lambda_1$ by $x_j := \mu'_j-\lambda'_j$. Since $|\lambda| = |\mu|$,
\[ \sum_{i=1}^t x_i =0 \]
where $t = \lambda_1$. If there exists some $j$ such that $x_j= 0$, then $(\lambda, \mu)$ is commonly reducible, a contradiction. Observe that $\lambda' \leq \mu'$ in dominance order (which is equivalent to $\lambda \geq \mu$ in dominance order) is equivalent to
\[ \sum_{i=1}^q x_i \geq 0 \text{ for all } 1 \leq q \leq t.\]
Thus $\vec{x}$ is a generalized Catalan sequence. Let $2y$ be the number of runs in $\vec{x}$ and let $x_{i_k}$ be the first element of the $k$th run ($1 \leq k \leq 2y$). Then the maximum number (in absolute value) of any $x_i's$ in run $k$, denoted as $m_k$ satisfies
\[ m_k \leq \max( \mu'_j \ : \ i_k \leq j < i_{k+1}) - \min(\lambda'_j \ : \ i_k \leq j < i_{k+1})  \leq \mu'_{i_k} - \lambda'_{i_{k+1}}\]
for $k$ odd and
\[ m_k \leq \max( \lambda'_j \ : \ i_k \leq j < i_{k+1}) - \min(\mu'_j \ : \ i_k \leq j < i_{k+1}) \leq \lambda'_{i_k} - \mu'_{i_{k+1}}\]
for $k$ even ($i_{2y+1} = t+1$). Hence
\[ cost(\vec{x}) = \sum_{k=1}^{2y} m_k \leq (\mu'_{i_1} - \lambda'_{i_2}) + (\lambda'_{i_2} - \mu'_{i_3}) + \cdots  = \mu'_{i_1} = \ell(\mu). \]
Therefore, due to the hypothesis $\lambda_1 = r$,
\[ \cost(\vec{x}) \leq \ell(\mu) \leq r = \lambda_1 = \width(\vec{x}). \]
Suppose $\cost(\vec{x}) < \width(\vec{x})$. Then by Theorem \ref{conj1}, $\vec{x}$ is reducible. In other words, there are two Catalan sublists $\vec{x}^{\bullet}$ and $\vec{x}^{\circ}$ which are complement to each other. Let $\vec{x}^{\bullet}$ correspond to the set of columns $C$ and $\vec{x}^{\circ}$ correspond to the set of columns $[\lambda_1] - C$. Let $\lambda^{\bullet}$ and $\lambda^{\circ}$ be partitions defined as columns $C$ and $[\lambda_1] - C$ of $\lambda$, respecitvely. Similarly define $\mu^{\bullet}$ and $\mu^{\circ}$. Since $\vec{x}^{\bullet}$ is Catalan, $\lambda^{\bullet} \geq \mu^{\bullet}$ in dominance order by the equivalence mentioned earlier in the proof. Similarly $\lambda^{\circ} \geq \mu^{\circ}$ in dominance order. Hence $(\lambda, \mu) = (\lambda^{\bullet}, \mu^{\bullet}) + (\lambda^{\circ}, \mu^{\circ})$ witnesses the common reducibility of $(\lambda, \mu)$, a contradiction.

Now $\cost(\vec{x}) = \width(\vec{x})$, which means that
\[ \cost(\vec{x}) = \ell(\mu) = r = \lambda_1 = \width(\vec{x}) .\]
If $y>1$, then by Theorem \ref{thm:equality} $\vec{x}$ is reducible. Then for the same reason as in the previous paragraph, $(\lambda, \mu)$ is commonly reducible, a contradiction. Similarly, if $y=1$ and $\vec{x}$ is reducible, then $(\lambda, \mu)$ is commonly reducible, also a contradiction.

Hence we have that $y=1$, $\cost(\vec{x}) = \width(\vec{x})$, and $\vec{x}$ is not reducible. Then by Theorem \ref{thm:y1}, all the entries of $\vec{x}$ are either $\alpha_1$ or $-\beta_1$, where $\alpha_1$ and $\beta_1$ are relatively prime. Thus
\[ x_1 = \cdots = x_{\beta_1} = \alpha_1, \quad x_{\beta_1+1} = \cdots =x_t =-\beta_1, \]
and $\alpha_1+\beta_1 = t = \lambda_1 = r$. Observe that
\[ x_{\beta_1} = \mu'_{\beta_1} - \lambda'_{\beta_1} = \alpha_1 \text{ and } x_{\beta_1+1} = \mu'_{\beta_1+1} - \lambda'_{\beta_1+1} = -\beta_1 .\]
This implies that
\[ \mu'_{\beta_1} = \lambda'_{\beta_1} + \alpha_1 \geq \lambda'_{\beta_1+1} + \alpha_1 = \mu'_{\beta_1+1} + \alpha_1 + \beta_1 = \mu'_{\beta_1+1} + r \geq r.\]
Since $\mu'_{\beta_1} \leq \ell(\mu) \leq r$, $\mu'_{\beta_1} = r$ and $\mu'_{\beta_1+1} = 0$. Therefore,
\[ \mu'_1 = \mu'_2 = \cdots = \mu'_{\beta_1} = r \text{ and } \mu'_{\beta_1+1} = \cdots = \mu'_t = 0 .\]
This means that
\[ \lambda'_1 = \lambda'_2 = \cdots = \lambda'_{\beta_1} = \beta_1 \text{ and } \lambda'_{\beta_1+1} = \cdots = \lambda'_t = \beta_1.\]
Thus $\lambda$ and $\mu$ are both rectangles. Moreover, $\lambda_1 = \alpha_1+\beta_1$ and $\mu_1 = \beta_1$ are relatively prime.
\end{proof}

\section*{Acknowledgment}
We thank Shiliang Gao, Joshua Kiers, Gidon Orelowitz, and Alexander Yong for helpful comments and discussions. We are
grateful to Shiliang Gao for raising the question that led to Theorem~\ref{thm:equality}.

\end{document}